\documentclass{elsarticle}
\usepackage{latexsym}
\usepackage{amsmath}
\usepackage{color} 
\usepackage{amssymb}
\usepackage{amsthm}
\usepackage{MnSymbol}
\usepackage{pgf,tikz}
\usepackage{caption}
\usepackage{chngcntr}
\usepackage{lipsum}
\usepackage{color}

\usepackage{pstricks-add}
\usepackage{tikz}
\usetikzlibrary{shapes.geometric}
\usetikzlibrary{arrows.meta,arrows}

\usepackage{pgf,tikz}
\usetikzlibrary{arrows}

\theoremstyle{plain}
\newtheorem{thm}{Theorem}[section]
\newtheorem{lem}{Lemma}[section]

\newtheorem{rem}{Remark}[section]
\newtheorem{prop}{Proposition}[section]

\newtheorem{Fact}{Fact}[section]

\newtheorem{Notation}{Notation}[section]

\numberwithin{equation}{section}

\begin{document}
\nocite{*}
\sloppy

\newtheorem*{Convention}{Convention}

\makeatletter
\def\ps@pprintTitle{%
	\let\@oddhead\@empty
	\let\@evenhead\@empty
	\def\@oddfoot{\reset@font\hfil\thepage\hfil}
	\let\@evenfoot\@oddfoot
}
\makeatother

\begin{frontmatter}

\title{Primarily orientable graphs}
\author{Houmem Belkhechine\fnref{}}
\address{University of Carthage, Bizerte Preparatory Engineering Institute, Bizerte, Tunisia}
\ead{houmem.belkhechine@ipeib.rnu.tn}

%\author{Cherifa Ben Salha\fnref{}}
%\address{University of Carthage, Faculty of Sciences of Bizerte, Bizerte, Tunisia}
%\ead{bensalha.cherifa1@gmail.com}
    \begin{abstract} 
    A graph $G$ is primarily orientable if it is possible to orient its edges in such a way that the resulting oriented graph is prime, i.e., indecomposable under modular decomposition. We characterize primarily orientable graphs.
    	\end{abstract}
\end{frontmatter}
\section{Introduction and main result}
An oriented graph is an orientation of a graph $G$, that is, a digraph obtained from $G$ by orienting each of its edges, i.e., by replacing each edge $\{u,v\}$ of $G$ by exactly one of the arcs $(u,v)$ and $(v,u)$. Both graphs and oriented graphs are considered in this paper. 
%Graph orientation is well-studied under different aspects and considerations (algebraic, %combinatorial, algorithmic, etc).
Graph orientation problems are well studied under different considerations involving various aspects (algebraic, combinatorial, algorithmic, etc.). Such studies have led several authors to a variety of interesting results revealing connections between properties of graphs and their orientations. For example, various connections between colorations and orientations of graphs have been established by several authors (e.g., see \cite{QW}). The general problem consists of characterizing those graphs having an orientation satisfying a given property. One of the earliest results concerning this problem dates back to 1939, when Robbins \cite{Robbins} proved that the strongly orientable graphs, i.e. the graphs admitting a strongly connected orientation, are the $2$-edge-connected graphs. Nash-Williams \cite{Nash} and Boesch and Tindel \cite{Boesch} generalized Robbins' result in different contexts. Another example of particular interest is that of comparability graphs (or transitively orientable graphs), i.e., graphs admitting transitive orientations. A first characterization of such graphs was obtained by Ghouila-Houri \cite{Ghouila} in 1962. In 1967, Gallai \cite{G67} characrerized these graphs in terms of forbidden subgraphs. Gallai's paper, originally in German, was translated to English by Maffray and Preissmann in $2000$ \cite{MP01}. Bisides comparability graphs, it contains the earliest fundamental results on the topic of modular decomposition.   

In this paper, we characterize the graphs that have a prime orientation, i.e., an orientation that is indecomposable under modular decomposition. We call such graphs primarily orientable. 

We now formalize our presentation.
A {\it graph} $G=(V(G),E(G))$ (resp. a {\it digraph} $G=(V(G),A(G))$) consists of a finite set $V(G)$ of {\it vertices} together with a set $E(G)$ (resp. $A(G)$) of unordered (resp. ordered) pairs of distinct vertices, called {\it edges} (resp. {\it arcs}). We write $v(G)$ for $|V(G)|$. For distinct $ u,v\in V(G)$, we write $\overrightarrow{uv}$ and $uv$ for the ordered and unordered pairs $(u,v)$ and $\{u,v\}$, respectively. Given a graph (resp. a digraph) $G$, with each subset $X$ of $V(G)$, we associate the {\it subgraph} $G[X] =(X, \binom{X}{2} \cap E(G))$ (resp. the {\it subdigraph} $G[X] =(X, (X \times X) \cap A(G))$) of $G$ induced by $X$. For $X\subseteq V(G)$ (resp. $x\in V(G)$), the subgraph or subdigraph $G[V(G) \setminus X]$ (resp. $G[V(G) \setminus \{x\}$]) is simply denoted by $G - X$ (resp. $G - x$). Two graphs $G$ and $G'$ are {\it isomorphic} if there exists an {\it isomorphism} from $G$ onto $G'$, that is, a bijection $f$ from $V(G)$ onto $V(G')$ such that for every $x,y\in V(G)$, $xy\in E(G)$ if and only if $f(x)f(y)\in E(G')$. 
%Let G be a graph. The {\it complement} of $G$, denoted by $\overline{G}$, is the graph defined by $V(\overline{G})=V(G)$ and $E(\overline{G})=\binom{V(G)}{ 2} \setminus E(G)$. 
The graph $G$ is {\it complete} if $E(G)=\binom{V(G)}{2}$, it is {\it edgeless} if $E(G)= \varnothing$. An $n$-vertex complete graph is denoted by $K_n$. 

An {\it orientation} of a graph $G$ is a digraph $\overrightarrow{G}$ obtained from $G$ by replacing each edge $uv$ of $G$ by exactly one of the ordered pairs $\overrightarrow{uv}$ and $\overrightarrow{vu}$ as an arc of $\overrightarrow{G}$, in such a way that for distinct $x,y \in V(G) = V(\overrightarrow{G})$,

\begin{equation*}
|A(\overrightarrow{G}) \cap \{\overrightarrow{xy}, \overrightarrow{yx}\}| \ = \
\begin{cases}
1 \ \ \ \text{if} \ xy \in E(G),\\

0 \ \ \ \text{if} \ xy \notin E(G).
\end{cases}
\end{equation*} 
An {\it oriented graph} is an orientation of some graph. For example, a {\it tournament} is an orientation of a complete graph.

We now introduce the notion of primality which is based on that of a module. It is convenient to use the following notations. 
\begin{Notation} \normalfont
	Given a graph or a digraph $G$, for $X \subseteq V(G)$, we write $\overline{X}$ for $V(G) \setminus X$.
\end{Notation}

\begin{Notation} \normalfont \label{notation}
	Let $G$ be a graph (resp. an oriented graph). For distinct $u,v \in V(G)$, we set  
	\begin{equation*}
	G(u,v)= \
	\begin{cases}
	1 \ \ \ \ \text{if} \ \  uv \in E(G),\\
	
	0 \ \ \ \   \text{if} \ \ uv \notin E(G)
	\end{cases}
	\end{equation*} 
	(resp.
	\begin{equation*}
	G(u,v)= \
	\begin{cases}
	\ 1 \ \ \ \ \ \text{if} \ \  \overrightarrow{uv} \in A(G),\\
	-1 \ \ \ \ \hspace{0.07cm} \text{if} \ \  \overrightarrow{vu} \in A(G),\\
	
	\ 0 \ \ \ \ \  \hspace{0.07cm}  \text{otherwise}.)
	\end{cases}
	\end{equation*} 
	Let $X$ and $Y$ be two disjoint subsets of $V(G)$. The notation $X \equiv_G Y$ signifies that $G(x,y) = G (x',y')$ for every $x, x' \in X$ and $y, y' \in Y$. For more precision when $X \equiv_G Y$, we write $G(X,Y) =0$ (resp. $G(X,Y) =1$, $G(X,Y) =-1$) to indicate that for every $x \in X$ and $y \in Y$, we have $G(x,y) =0$ (resp. $G(x,y) =1$, $G(x,y) =-1$). When $X$ is a singleton $\{x\}$, we write $x \equiv_G Y$ for $\{x\} \equiv_G Y$, $G(x,Y)$ for $G(\{x\}, Y)$, and $G(Y,x)$ for $G(Y, \{x\})$. 
\end{Notation}

  Given a graph (resp. a digraph) $G$, a subset $M$ of $V(G)$ is a {\it module}  \cite{CH93, S92} (or a {\it homogeneous subset}~\cite{G67,MP01} or an {\it interval} \cite{ST}) of $G$  provided that for every $x,y\in M$ and for every $v\in \overline{M}$, $vx \in E(G)$ if and only if $vy \in E(G)$ (resp. $\overrightarrow{vx} \in A(G)$ if and only if $\overrightarrow{vy} \in A(G)$, and $\overrightarrow{xv} \in A(G)$ if and only if $\overrightarrow{yv} \in A(G)$). For example, by using Notation~\ref{notation}, when $G$ is a graph or an oriented graph, a subset $M$ of $V(G)$ is a module of $G$ if and only if for every $v \in \overline{M}$, we have $v \equiv_G M$, i.e., either $G(v,M) = 0$, $G(v,M) = 1$, or $G(v,M) = -1$.
  \begin{Notation} \normalfont
  	Given a graph or a digraph $G$, the set of modules of $G$ is denoted by $\text{mod}(G)$.
  \end{Notation} 
For example, it follows from the definition of modules, that
\begin{equation} \label{eqmodGGfleche}
\text{if} \ \overrightarrow{G} \ \text{is an orientation of a graph} \ G, \ \text{then} \ \text{mod}(\overrightarrow{G}) \subseteq \text{mod}(G).
\end{equation}

  Let $G$ be a graph or a digraph. The sets $V(G)$, $\varnothing$ and $\{v\}$ ($v \in V(G)$) are modules of $G$ called {\it trivial modules}. We say that $G$ is {\it indecomposable} \cite{ST}  if all its modules are trivial, otherwise we say that $G$ is {\it decomposable}. For example, the graphs and digraphs with at most two vertices are indecomposable. So let us say that a graph (or a digraph) is {\it prime} \cite{CH93} if it is indecomposable with at least three vertices. 
  
  A graph is {\it primarily orientable} if it admits a prime orientation. For instance, it follows from (\ref{eqmodGGfleche}) that every orientation of a prime graph is prime as well. In particular, prime graphs are primarily orientable. But primality is not a necessary condition for a graph to be primarily orientable. For example, for every integer $n \geq 5$, the complete graph $K_n$, which is obviously decomposable, is primarily orientable because there exist prime tournaments with $n$ vertices (e.g., see \cite{HB}). In fact, almost all tournaments are prime \cite{Erdos}. However, $K_4$ is not primarily orientable because the $4$-vertex tournaments are decomposable (e.g., see \cite{HC}).

  The main result of the paper consists of a characterization of primarily orientable graphs. In order to present this characterization, we need to introduce the following notions. A graph $G$ is {\it connected} if for every distinct $x, y \in V(G)$, there exists a sequence $v_0 =x, \ldots, v_n =y$ of vertices such that for every $i \in \{0, \ldots, n-1\}$, $v_iv_{i+1} \in E(G)$. A \textit{connected component} of the graph $G$ is a maximal subset $C$ of $V(G)$ such that $G[C]$ is connected. The connected components of $G$ form a partition of $V(G)$.
  \begin{Notation} \normalfont
  	Given a graph $G$, the set of the connected components of $G$ is denoted by $\mathcal{C}(G)$.
  \end{Notation} 
  Let $C \in \mathcal{C}(G)$. We have $G(C, \overline{C}) = 0$, and hence $C$ and $\overline{C}$ are modules of $G$. Similarly for every orientation $\overrightarrow{G}$ of $G$, $C$ and $\overline{C}$ are modules of $\overrightarrow{G}$ because $\overrightarrow{G}(C, \overline{C}) = G(C, \overline{C}) = 0$. Thus, given a graph $G$ with at least three vertices,
  \begin{equation} \label{eqconnecprimorient}
  \text{if the graph} \ G \ \text{is primarily orientable, then} \  G \ \text{is connected}.
  \end{equation}  
  But connectedness is not a sufficient condition for a graph to be primarily orientable. For example, for every integer $n \geq 3$, the graph $K_{1,n} = (\{0, \ldots, n\}, \{0i : i \in \{1, \ldots,n\}\})$, called a {\it star}, is connected but not primarily orientable.

  Let $G$ be a graph. A {\it stable set} of $G$ is a subset $S$ of $V(G)$ such that $|S| \geq 2$ and the induced subgraph $G[S]$ is edgeless. Let us say that a subset $S$ of $V(G)$ is a {\it stable module} (or a {\it s-module}) of $G$ if $S$ is a module of $G$ which is also a stable set of $G$.
  
  Given a vertex $x$ of the graph $G$, the {\it neighbourhood} of $x$ in $G$, denoted by $N_{G}(x)$, is the set of the vertices 
  $y\in V(G)$ such that $xy\in E(G)$. The {\it degree} of the vertex $x$ (in $G$) is  $d_G(x) = |N_{G}(x)|$.
  We extend this to nonempty modules in the following manner. Let $M$ be a nonempty module of the graph $G$. We define the {\it neighbourhood} of $M$ in $G$, denoted by $N_G(M)$, to be the set of the vertices $v \in \overline{M}$ satisfying one of the following equivalent conditions :
  \begin{itemize}
  	\item $vx \in E(G)$ for every $x \in M$,
  	\item $vx \in E(G)$ for some $x \in M$.
  \end{itemize}
The {\it degree} of $M$ (in $G$) is $d_G(M) = |N_{G}(M)|$.

Our characterization of primarily orientable graphs is provided by the following theorem, which is the main result of the paper. 
  %The following theorem is the main result of the paper.
  
  \begin{thm} \label{thmprincipal}
  	Given a graph $G$ with at least three vertices, other than $K_4$, the graph $G$ is primarily orientable if and only if the following two conditions are satisfied.
  	\begin{enumerate}
  		\item The graph $G$ is connected.
  		
  		\item For every s-module $S$ of $G$, we have $d_G(S) \geq \log_2(|S|)$.
  	\end{enumerate}
  	\end{thm}    
  
  \section{Preliminary results}
  %We begin with some useful properties of modules.
  %\begin{prop} \label{propo modules}
  %	Let $G$ be a graph or a digraph.
  	%\begin{enumerate}
  	%	\item Given a subset $W$ of $V(G)$, if $M$ is a module of $G$, then $M \cap W$ is a module of $G[W]$.
  		%\item Given a module $M$ of $G$, if $N$ is a module of $G[M]$, then $N$ is also a module of $G$.
  		%\item If $M$ and $N$ are modules of $G$, then $M \cap N$ is also a module of $G$.
  	%	\item If $M$ and $N$ are modules of $G$ such that $M \cap N \neq \varnothing$, then $M \cup N$ is also a module of $G$.
  		%\item If $M$ and $N$ are modules of $G$ such that $M \setminus N \neq \varnothing$, then $N \setminus M$ is also a module of $G$.
  		%\item If $M$ and $N$ are disjoint modules of $G$, then $M \equiv_G N$.
  	%\end{enumerate}
  %\end{prop}   

  \begin{lem} \label{lemrem}
  	Given a connected graph $G$, if $M$ is a nontrivial module of $G$, then for every vertex $v \in M$, the graph $G-v$ is connected as well.
  \end{lem}

\begin{proof}
	Let $G$ be a connected graph, let $M$ be a nontrivial module of $G$, and let $v \in M$. Since $G$ is connected, we have 
	\begin{equation} \label{eqd}
	d_G(M) \geq 1.
	\end{equation}
	Moreover, for every $C \in \mathcal{C}(G-v)$, there exists a vertex $v_C \in C$ such that $vv_C \in E(G)$. Suppose for a contradiction that $G-v$ is not connected. If there exists $C \neq C' \in \mathcal{C}(G-v)$ such that $v_C \in M$ and $v_{C'} \notin M$, then $G(v_C,v_{C'}) = 0 \neq G(v, v_{C'}) =1$, which contradicts that $M$ is a module of $G$. Therefore, either $\{v_C : C \in \mathcal{C}(G-v)\} \subseteq M$ or $M \cap \{v_C : C \in \mathcal{C}(G-v)\} = \varnothing$. First suppose $\{v_C : C \in \mathcal{C}(G-v)\} \subseteq M$. Let $u \in N_G(M)$ (see (\ref{eqd})), and let $C$ be the connected component of $G-v$ containing $u$. For $C' \in \mathcal{C}(G-v) \setminus \{C\}$, we obtain $G(u,v_{C'}) = 0 \neq G(u,v) = 1$, which contradicts that $M$ is a module of $G$. Second suppose $M \cap \{v_C : C \in \mathcal{C}(G-v)\} = \varnothing$. Let $w \in M \setminus \{v\}$ and let $C$ be the connected component of $G-v$ containing $w$. For $C' \in \mathcal{C}(G-v) \setminus \{C\}$, we obtain $G(v_{C'},w) =0 \neq G(v_{C'},v) = 1$, which again contradicts that $M$ is a module of $G$.  
\end{proof}
  
Given a graph $G$, a {\it duo} of $G$ is a module of cardinality 2 of $G$. We say that a subset $X$ of $V(G)$ is a {\it stable duo} (or a {\it s-duo}) of $G$ if $X$ is a duo of $G$ which is also a stable set of $G$.

\begin{Notation} \normalfont
	Given a graph $G$, the set of s-modules (resp. s-duos) of $G$ is denoted by $\text{smod}(G)$ (resp. $\text{sduo}(G)$).
	%Let $G$ be a graph or a digraph. We write $\text{smod}(G)$ for the set of s-modules $S$ of $G$ such that $|S| \geq 2$. Similarly, the set of s-duos of $G$ is denoted by $\text{sduo}(G)$.  
\end{Notation} 
The next remarks follow immediately from the definitions of s-modules and s-duos. 

\begin{rem} \normalfont \label{remsubsmodule}
	If $S$ is a s-module of a graph $G$, then every subset $S'$ of at least two elements of $S$ is also a s-module of $G$.
\end{rem}

\begin{rem} \normalfont \label{remunionsmodules}
	Given two s-modules $S$ and $S'$ of a graph $G$, if $S \cap S' \neq \varnothing$, then $S \cup S'$ is also a s-module of $G$.
\end{rem}
\begin{rem} \normalfont \label{rem}
	For every graph $G$, we have $\cup \text{smod}(G) = \cup \text{sduo}(G)$.
\end{rem} 
\begin{lem} \label{leminclustrict}
	Given a graph $G$, for every $x \in \cup {\rm smod}(G) = \cup {\rm sduo}(G)$ (see Remark~\ref{rem}), we have ${\rm sduo}(G-x) \varsubsetneq {\rm sduo}(G)$.     
\end{lem}
\begin{proof}
	Let $G$ be a graph and let $x \in \cup \text{sduo}(G)$. Let $uv \in \text{sduo}(G-x)$, where $u \neq v \in V(G-x)$. To prove that $uv \in \text{sduo}(G)$, it suffices to show that $G(x,u) = G(x,v)$. Consider a vertex $y \in V(G-x)$ such that $xy \in \rm{sduo}(G)$. First suppose $y \notin uv$. We have $G(y,u) = G(y,v)$ because $uv$ is a module of $G-x$. In addition, $G(y,u) = G(x,u)$ and $G(y,v) = G(x,v)$ because $xy$ is a module of $G$. Thus $G(x,u) = G(x,v)$. Second suppose $y \in uv$. We may assume $y=u$. In this instance, $G(x,u) = G(u,v) =0$ because $xu$ and $uv$ are stable sets of $G$. Moreover, since $G(u,v) =0$ and $xu$ is a module of $G$, we obtain $G(x,v) =0$. Thus $G(x,u) = G(x,v)$. We conclude that $\text{sduo}(G-x) \subseteq \text{sduo}(G)$. Since $xy \in \text{sduo}(G) \setminus \text{sduo}(G-x)$, it follows that $\text{sduo}(G-x) \varsubsetneq \text{sduo}(G)$.   
\end{proof} 
Let us say that a graph $G$ is {\it sduo-free} when $\text{sduo}(G) = \varnothing$.

\begin{lem} \label{lemsduo}
	Given a sduo-free graph $G$, for every $v \in V(G)$, the elements of ${\rm sduo}(G-v)$ are pairwise disjoint. 
\end{lem}
\begin{proof}
Let $G$ be a sduo-free graph and let $v \in V(G)$. Let $S$ and $S'$ be two distinct elements of ${\rm sduo}(G-v)$. Suppose for a contradiction that $S \cap S' \neq \varnothing$. In this instance, there are three pairwise distinct vertices $x,y,z \in V(G-v)$ such that $S = xy$ and $S' = yz$. There are distinct $\alpha, \beta \in \{x,y,z\}$ such that $G(v, \alpha) = G(v, \beta)$. On the other hand, since $S \cup S' = \{x,y,z\} \in \text{smod}(G-v)$ (see Remark~\ref{remunionsmodules}), then $\alpha\beta \in \text{sduo}(G-v)$ (see Remark~\ref{remsubsmodule}). Since $G(v, \alpha) = G(v, \beta)$, it follows that $\alpha\beta \in \text{sduo}(G)$, contradicting that $G$ is sduo-free. 
\end{proof}
 
% \begin{lem} \label{lemsduo}
 %	Given a sduo-free graph $G$, for every $v \in V(G)$, we have ${\rm smod}(G-v) = {\rm sduo}(G-v)$. In particular, the elements of ${\rm sduo}(G-v)$ are pairwise disjoint. 
 %\end{lem} 

%\begin{proof}
%	Let $G$ be a sduo-free graph and let $v \in V(G)$. Clearly $\text{sduo}(G-v) \subseteq \text{smod}(G-v)$. To prove that $\text{smod}(G-v) \subseteq \text{sduo}(G-v)$, let $S \in \text{smod}(G-v)$. Suppose for a contradiction that $|S| \geq 3$. In this instance, there exist distinct $x, y \in S$ such that $G(v,x) = G(v,y)$. Since $S \in \text{smod}(G-v)$, it follows that $xy \in \text{sduo}(G)$, which contradicts that the graph $G$ is sduo-free. Therefore $|S|=2$, i.e. $S \in \text{sduo}(G-v)$. We conclude that $\text{smod}(G-v) = \text{sduo}(G-v)$. 
	
%	Now let $D$ and $D'$ be two distinct elements of $\text{sduo}(G-v)$. If $D \cap D' \neq \varnothing$, then $|D \cup D'| =3$ and $D \cup D' \in \text{smod}(G-v)$ (see Remark~\ref{remunionsmodules}), which contradicts $\text{smod}(G-v) = \text{sduo}(G-v)$. Therefore, the elements of ${\rm sduo}(G-v)$ are pairwise disjoint.           
%\end{proof}

In the proof of the next lemma, we need the following well-known fact (e.g., see \cite[Exercise 2.7.6]{Bondy}).
\begin{Fact} \label{fact}
	Given a connected graph $G$ such that $v(G) \geq 1$, there exists a vertex $v$ of $G$ such that $G-v$ is connected as well.  
\end{Fact}

\begin{lem} \label{lemauplus1}
	Let $G$ be a graph such that $v(G) \geq 1$. If $G$ is connected and sduo-free, then there exists $v \in V(G)$ such that $G-v$ is connected and $|{\rm sduo}(G-v)| \leq 1$. 
\end{lem}
  
\begin{proof}
	Suppose that the graph $G$ is connected and sduo-free. By Fact~\ref{fact}, there exists a vertex $x \in V(G)$ such that $G-x$ is connected. If $|\text{sduo}(G-x)| \leq 1$, then we are done. Hence suppose $|\text{sduo}(G-x)| \geq 2$. By Lemma~\ref{lemsduo}, there are four pairwise dictint vertices $v_1, v_2, v_3, v_4 \in V(G-x)$ such that $v_1v_2$ and $v_3v_4$ are s-duos of $G-x$. Since the graph $G$ is sduo-free, we may assume 
	\begin{equation} \label{eq0}
	G(x,v_1v_3) = 0 \ \text{and} \ G(x,v_2v_4) = 1.
	\end{equation}
	Since $G-xv_1 = (G-x)-v_1$ is connected by Lemma~\ref{lemrem}, and since $xv_2 \in E(G)$ (see (\ref{eq0})), it follows that $G-v_1$ is connected as well. We will prove that the graph $G- v_1$ is sduo-free. We first prove that 
	\begin{equation} \label{eq1}
	\text{sduo}(G-v_1x) \cap \text{sduo}(G-v_1) = \varnothing.
	\end{equation}
	
	Let $S = ss' \in \text{sduo}(G - v_1x)$, where $s \neq s' \in V(G-v_1x)$. To show that $S \in \text{sduo}(G-x)$, it suffices to verify that $G(v_1,s) = G(v_1,s')$. To begin, suppose $v_2 \in ss'$. We may assume $s=v_2$. In this instance, $G(v_1,s') = G(v_2,s') = G(s,s') = 0$ because $v_1v_2$ is a module of $G-x$ and $ss'$ is a stable set of $G$. Since $G(v_1,s) = G(v_1,v_2) = 0$ because $v_1v_2$ is a stable set of $G$, it follow that $G(v_1,s) = G(v_1,s')$ as desired. Now suppose $v_2 \notin ss'$. We have $G(v_1,s) = G(v_2,s)$ and $G(v_1,s') = G(v_2,s')$ because $v_1v_2$ is a module of $G-x$. Moreover, $G(v_2,s) = G(v_2,s')$ because $ss'$ is a module of $G - v_1x$. Thus $G(v_1,s) = G(v_1,s')$, and hence $S \in \text{sduo}(G-x)$. Since $G$ is sduo-free, it follows that $G(x,s) \neq G(x,s')$, and hence $S \notin \text{sduo}(G-v_1)$. Thus (\ref{eq1}) holds.
	
	Now suppose for a contradiction that the graph $G-v_1$ admits a s-duo $D$. It follows from (\ref{eq1}) that $x \in D$. Therefore, $D = xy$ for some $y \in \overline{v_1x}$. To show that $y \in \overline{ \{x,v_1,v_2,v_3,v_4\}}$, we have to verify that $y \notin \{v_2,v_3,v_4\}$. We have $y \notin v_2v_4$ because $G(x,y) =0 \neq G(x, v_2v_4) = 1$ (see (\ref{eq0})). Since $G(y,v_4) = G(x,v_4)$ because $xy$ is module of $G-v_1$, and since $G(x,v_4) =1$ by (\ref{eq0}), we obtain $G(y,v_4) = 1$, and hence $y \neq v_3$ because $G(v_3,v_4) = 0$. Thus $y \in \overline{\{x,v_1,v_2,v_3,v_4\}}$. Since $G(x,v_4) = 1$ (see (\ref{eq0})), then $G(y,v_4) = 1$ because $xy$ is a module of $G-v_1$, and hence $G(y,v_3) = 1$ because $v_3v_4$ is a module of $G-x$. Thus $G(y,v_3) =1 \neq G(x,v_3) = 0$ (see (\ref{eq0})), contradicting that $xy$ is a module of $G-v_1$. Hence $G-v_1$ is sduo-free. Since $G-v_1$ is connected, this completes the proof.  
	\end{proof}  
  
  \begin{rem} \normalfont
  	Lemma~\ref{lemauplus1} does not hold if we replace $|{\rm sduo}(G-v)| \leq 1$ by sduo-free. For example, consider the {\it half graph} \cite{EH} $G_{2n}$ defined on $V(G_{2n}) = \{0, \ldots, 2n-1\}$, where $n \geq 2$, by $E(G_{2n}) = \{(2i)(2j+1) : 0 \leq i \leq j \leq n-1\}$ (see Figure~\ref{G2n}). For every integer $n \geq 2$, the graph $G_{2n}$ is prime, in particular it is connected and sduo-free. In fact, the graphs $G_{2n}$ are, up to isomorphim, the {\it critical} graphs, i.e., the prime graphs $G$ such that for every $v \in V(G)$, the graph $G-v$ is decomposable (see e.g., \cite{ST, BI09}). The graphs $G_{2n}-0$ and $G_{2n}-(2n-1)$ are not connected. Moreover, for every vertex $i \in \{1, \ldots, 2n-2\}$, $\{i-1, i+1\}$ is a s-duo of $G_{2n}-i$. Thus, there does not exist a vertex $v$ of $G_{2n}$ such that $G_{2n}-v$ is connected and sduo-free. 
  \end{rem}
 
  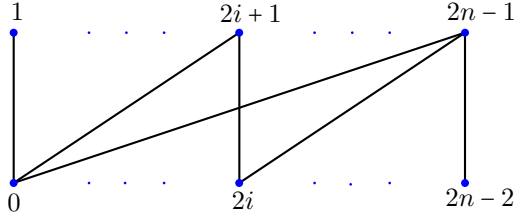
\begin{figure}[h]
  	\begin{center}
  \psset{xunit=1.0cm,yunit=1.0cm,algebraic=true,dotstyle=o,dotsize=3pt 0,linewidth=0.8pt,arrowsize=3pt 2,arrowinset=0.25}
  \begin{pspicture*}(-0.84,0.32)(6.88,3.72)
  \psline(3,1)(3,3)
  \psline(6,1)(6,3)
  \psline(3,1)(6,3)
  \psline(0,3)(0,1)
  \psline(0,1)(3,3)
  \psline(0,1)(6,3)
  \rput[tl](-0.04,3.4){$1$}
  \rput[tl](-0.08,0.86){$0$}
  \rput[tl](2.9,0.88){$2i$}
  \rput[tl](2.74,3.38){$2i+1$}
  \rput[tl](5.74,0.92){$2n-2$}
  \rput[tl](5.76,3.4){$2n-1$}
  \begin{scriptsize}
  \psdots[dotstyle=*,linecolor=blue](0,3)
  \psdots[dotstyle=*,linecolor=blue](3,1)
  \psdots[dotstyle=*,linecolor=blue](3,3)
  \psdots[dotstyle=*,linecolor=blue](6,1)
  \psdots[dotstyle=*,linecolor=blue](6,3)
  \psdots[dotsize=1pt 0,dotstyle=*,linecolor=blue](1,1)
  \psdots[dotsize=1pt 0,dotstyle=*,linecolor=blue](1.48,1)
  \psdots[dotsize=1pt 0,dotstyle=*,linecolor=blue](2,1)
  \psdots[dotsize=1pt 0,dotstyle=*,linecolor=blue](1,3)
  \psdots[dotsize=1pt 0,dotstyle=*,linecolor=blue](2,3)
  \psdots[dotsize=1pt 0,dotstyle=*,linecolor=blue](1.5,3)
  \psdots[dotsize=1pt 0,dotstyle=*,linecolor=blue](4,1)
  \psdots[dotsize=1pt 0,dotstyle=*,linecolor=blue](5,1)
  \psdots[dotsize=1pt 0,dotstyle=*,linecolor=blue](5,3)
  \psdots[dotsize=1pt 0,dotstyle=*,linecolor=blue](4,3)
  \psdots[dotsize=1pt 0,dotstyle=*,linecolor=blue](4.48,0.98)
  \psdots[dotsize=1pt 0,dotstyle=*,linecolor=blue](4.5,3)
  \psdots[dotstyle=*,linecolor=blue](0,1)
  \end{scriptsize}
  \end{pspicture*}
  \caption{The half graph $G_{2n}$.}
  \label{G2n}
  \end{center}
  \end{figure}

  \begin{Notation} \normalfont \label{notationpGX}
  	Let $G$ be a graph or a digraph. With any subset $X$ of $V(G)$ such that $G[X]$ is prime, we associate the following subsets of $\overline{X}$. 
  	\begin{itemize}
  		\item $\text{Ext}_G(X)$ is the set of $v \in \overline{X}$ such that $G[X \cup \{v\}]$ is prime;
  		\item $\langle X \rangle_G$ is the set of $v \in \overline{X}$ such that $X$ is a module $G[X \cup \{v\}]$;
  		\item given $u \in X$, $X_G(u)$ is the set of $v \in \overline{X}$ such that $uv$ is a module of $G[X \cup \{v\}]$.
  	\end{itemize}
  The family $\{\text{Ext}_G(X), \langle X \rangle_G\} \cup \{X_G(u) : u \in X\}$ is denoted by $p_{(G, \overline{X})}$.
   \end{Notation}
  In fact, 
  \begin{equation} \label{eqpGX}
  \text{the family} \ p_{(G, \overline{X})} \ \text{form a partition of} \ \overline{X} \ \text{(see Lemma~\ref{lemER})}. 
  \end{equation}
  It is used as follows.
  
  \begin{lem} [see Lemmas 6.3 and 6.4 of \cite{ER}] \label{lemER}
  	Given a graph or a digraph $G$ with a subset $X$ of $V(G)$ such that $G[X]$ is prime, the family $p_{(G, \overline{X})}$ form a partition of $\overline{X}$. Moreover, the following three assertions hold.
  	\begin{enumerate}
  		\item For $u \in \langle X \rangle_G$ and $v \in \overline{X} \setminus \langle X \rangle_G$, if $X \cup \{v\}$ is not a module of $G[X \cup uv]$, then $G[X \cup uv]$ is prime.
  		\item Given $u \in X$, for $x \in X_G(u)$ and $y \in \overline{X} \setminus X_G(u)$, if $ux$ is not a module of $G[X \cup xy]$, then $G[X \cup xy]$ is prime.
  		\item For distinct $u,v \in {\rm Ext}_G(X)$, if $uv$ is not a module of $G[X \cup uv]$, then $G[X \cup uv]$ is prime.
  		\end{enumerate} 
  \end{lem}

  \section{Proof of Theorem~\ref{thmprincipal}}
  Our proof of Theorem~\ref{thmprincipal} is by induction on $|\text{sduo}(G)|$. Proposition~\ref{propsduofree} below is then the basis step. We also need to verify the prime  orientability of some graphs of small sizes (see Facts~\ref{factsize34} and \ref{fact5verices}). Since such verifications are easy to do, we omit the details.   
  
\begin{Fact} \label{factsize34}
	Up to isomorphism, there are exactly six graphs with three or four vertices satisfying Conditions $1$ and $2$ of Theorem~\ref{thmprincipal} and other than $K_4$. Each of these (connected) graphs is primarily orientable. Figure~\ref{fig34} below gives a prime orientation of each of them.  
	%the graphs with three or four vertices other than $K_4$ and satisfying Conditions 1 and 2 of Theorem~\ref{thmprincipal} are $K_3$, $P_3 = (\{0,1,2\}, \{01,12\})$, $P_4 = (\{0,1,2,3\}, \{01, 12, 23\})$, $C_4 = (\{0, 1, 2, 3\}, \{01, 12, 23, 30\})$, $Q_4 = (\{0, 1, 2, 3\}, \{01, 12, 23, 30, 02\})$, and $R_4 = (\{0, 1, 2, 3\}, \{01, 02, 12, 03\})$. Each of these six graphs is primarily orientable.  
\begin{figure}[h]
	\begin{center}
	\psset{xunit=1.0cm,yunit=1.0cm,algebraic=true,dotstyle=o,dotsize=3pt 0,linewidth=0.8pt,arrowsize=3pt 2,arrowinset=0.25}
	\begin{pspicture*}(-4.34,2.6)(7.5,4.26)
	\psline{->}(-3.52,3.8)(-4,3)
	\psline{->}(-3,3)(-3.52,3.8)
	\psline{->}(-1,3)(-1.52,3.78)
	\psline{->}(-1.52,3.78)(-2,3)
	\psline{->}(-2,3)(-1,3)
	\psline{->}(1,3)(1,4)
	\psline{->}(1,4)(0,4)
	\psline{->}(0,4)(0,3)
	\psline{->}(0,3)(1,3)
	\psline{->}(3,3)(3,4)
	\psline{->}(3,4)(2,4)
	\psline{->}(2,4)(2,3)
	\psline{->}(2,3)(3,3)
	\psline{->}(3,3)(2,4)
	\psline{->}(4,3)(5,3)
	\psline{->}(5,3)(4,4)
	\psline{->}(4,4)(4,3)
	\psline{->}(5,3)(5,4)
	\psline{->}(7,3)(7,4)
	\psline{->}(7,4)(6,4)
	\psline{->}(6,4)(6,3)
	\begin{scriptsize}
	\psdots[dotstyle=*,linecolor=blue](-4,3)
	\psdots[dotstyle=*,linecolor=blue](-3.52,3.8)
	\psdots[dotstyle=*,linecolor=blue](-3,3)
	\psdots[dotstyle=*,linecolor=blue](-2,3)
	\psdots[dotstyle=*,linecolor=blue](-1.52,3.78)
	\psdots[dotstyle=*,linecolor=blue](-1,3)
	\psdots[dotstyle=*,linecolor=blue](0,3)
	\psdots[dotstyle=*,linecolor=blue](1,3)
	\psdots[dotstyle=*,linecolor=blue](1,4)
	\psdots[dotstyle=*,linecolor=blue](0,4)
	\psdots[dotstyle=*,linecolor=blue](2,3)
	\psdots[dotstyle=*,linecolor=blue](2,4)
	\psdots[dotstyle=*,linecolor=blue](3,4)
	\psdots[dotstyle=*,linecolor=blue](3,3)
	\psdots[dotstyle=*,linecolor=blue](4,3)
	\psdots[dotstyle=*,linecolor=blue](5,3)
	\psdots[dotstyle=*,linecolor=blue](5,4)
	\psdots[dotstyle=*,linecolor=blue](4,4)
	\psdots[dotstyle=*,linecolor=blue](6,3)
	\psdots[dotstyle=*,linecolor=blue](7,3)
	\psdots[dotstyle=*,linecolor=blue](7,4)
	\psdots[dotstyle=*,linecolor=blue](6,4)
	\end{scriptsize}	
	\end{pspicture*}
	\caption{Prime orientations of the primarily orientable graphs $G$ with $v(G) \leq 4$.}
	\label{fig34}
	\end{center}
\end{figure}
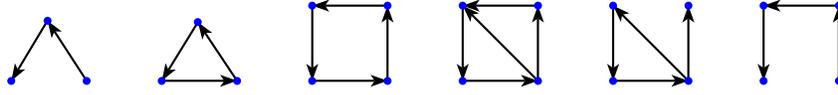
\end{Fact}

\begin{Fact} \label{fact5verices}
	Up to isomorphism, there are exactly four connected graphs $G$ with a vertex $v \in V(G)$ such that $G-v$ is $K_4$. Each of these (connected) graphs is primarily orientable. See Figure~\ref{fig5} for a prime orientation of each of them. 
	%Given a graph $G$ with five vertices, if there exists $v \in V(G)$ such that $G-v$ is $K_4$ and $d_G(v) \geq 1$, then $G$ is primarily orientable.
	\begin{figure}[h]
		\begin{center}
\psset{xunit=1.0cm,yunit=1.0cm,algebraic=true,dotstyle=o,dotsize=3pt 0,linewidth=0.8pt,arrowsize=3pt 2,arrowinset=0.25}
\begin{pspicture*}(-3.64,1.52)(7.76,4.32)
\psline{->}(-3,2)(-2,2)
\psline{->}(-2,2)(-2,3)
\psline{->}(-2,3)(-3,3)
\psline{->}(-3,3)(-3,2)
\psline{->}(-2,3)(-3,2)
\psline{->}(-3,3)(-2,2)
\psline{->}(-3,3)(-2.5,3.9)
\psline{->}(0,2)(1,2)
\psline{->}(1,2)(1,3)
\psline{->}(1,3)(0,3)
\psline{->}(0,3)(0,2)
\psline{->}(0,3)(1,2)
\psline{->}(1,3)(0,2)
\psline{->}(0,3)(0.48,3.9)
\psline{->}(0.48,3.9)(1,3)
\psline{->}(4,2)(4.24,2.88)
\psline{->}(4.24,2.88)(3.48,3.58)
\psline{->}(3.48,3.58)(2.82,2.88)
\psline{->}(2.82,2.88)(3,2)
\psline{->}(3,2)(4.24,2.88)
\psline{->}(4,2)(3.48,3.58)
\psline{->}(4.24,2.88)(2.82,2.88)
\psline{->}(3.48,3.58)(3,2)
\psline{->}(2.82,2.88)(4,2)
\psline{->}(6,2)(7,2)
\psline{->}(7,2)(7.22,2.86)
\psline{->}(7.22,2.86)(6.48,3.56)
\psline{->}(6.48,3.56)(5.78,2.86)
\psline{->}(5.78,2.86)(6,2)
\psline{->}(6,2)(7.22,2.86)
\psline{->}(7,2)(6.48,3.56)
\psline{->}(7.22,2.86)(5.78,2.86)
\psline{->}(6.48,3.56)(6,2)
\psline{->}(5.78,2.86)(7,2)
\begin{scriptsize}
\psdots[dotstyle=*,linecolor=blue](-3,2)
\psdots[dotstyle=*,linecolor=blue](-2,3)
\psdots[dotstyle=*,linecolor=blue](-3,3)
\psdots[dotstyle=*,linecolor=blue](0,3)
\psdots[dotstyle=*,linecolor=blue](0,2)
\psdots[dotstyle=*,linecolor=blue](1,2)
\psdots[dotstyle=*,linecolor=blue](1,3)
\psdots[dotstyle=*,linecolor=blue](-2,2)
\psdots[dotstyle=*,linecolor=blue](-2.5,3.9)
\psdots[dotstyle=*,linecolor=blue](0.48,3.9)
\psdots[dotstyle=*,linecolor=blue](3.48,3.58)
\psdots[dotstyle=*,linecolor=blue](4,2)
\psdots[dotstyle=*,linecolor=blue](4.24,2.88)
\psdots[dotstyle=*,linecolor=blue](2.82,2.88)
\psdots[dotstyle=*,linecolor=blue](3,2)
\psdots[dotstyle=*,linecolor=blue](6,2)
\psdots[dotstyle=*,linecolor=blue](7,2)
\psdots[dotstyle=*,linecolor=blue](7.22,2.86)
\psdots[dotstyle=*,linecolor=blue](5.78,2.86)
\psdots[dotstyle=*,linecolor=blue](6.48,3.56)
\end{scriptsize}
\end{pspicture*}
	\caption{Prime orientations of the four graphs described in Fact~\ref{fact5verices}.}
	\label{fig5}
	\end{center}
	 \end{figure}
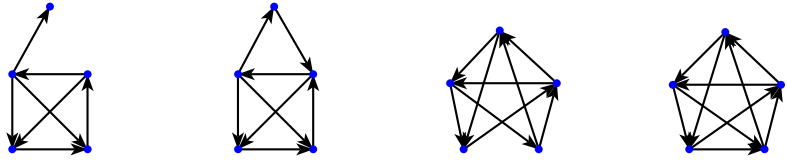
\end{Fact}

%\begin{Fact} \label{fact6vertices}
	%Given a graph $G$ with six vertices, if there exists $x \in V(G)$ such that $G-x$ is the graph obtained from $K_5$ by deleting one edge $uv$, where $u \neq v \in V(K_5)$, and if $G(x,u) \neq G(x,v)$, then $G$ is primarily orientable. 
%\end{Fact}

  \begin{prop} \label{propsduofree}
  	Given a graph $G$ with at least three vertices, other than $K_4$, if $G$ is connected and sduo-free, then $G$ is primarily orientable.
  	%	Every connected and sduo-free graph, other than $K_4$, is primarily orientable.
  \end{prop}
\begin{proof}
	Consider a connected and sduo-free graph $G$ other than $K_4$ and such that $v(G) \geq3$. To prove that $G$ is primarily orientable, we proceed by induction on $v(G) =n$. When $n =3$ or $4$, the graph $G$ is primarily orientable by Fact~\ref{factsize34}. Let $n \geq 5$. By Lemma~\ref{lemauplus1}, there exists $x \in V(G)$ such that $G-x$ is connected and $|\text{sduo}(G-x)| \leq 1$. 
	
	First suppose $|\text{sduo}(G-x)| = 1$. Let $S = ss'$ be the unique s-duo of $G-x$. Since $S$ is a s-duo of $G-x$ but not of $G$, we have $G(x,s) \neq G(x,s')$. By interchanging $s$ and $s'$, we may assume $G(x,s) =0$ and $G(x,s') =1$. By Lemma~\ref{lemrem}, the subgraph $G-xs = (G-x)-s$ is connected. Moreover, $G-xs$ is sduo-free by Lemma~\ref{leminclustrict}. By the induction hypothesis, $G-xs$ is $K_4$ or $G-xs$ is primarily orientable. Suppose that $G -xs$ is $K_4$. In this instance, $G-x$ is the graph obtained from $K_5$ by deleting the edge $ss'$. By Fact~\ref{fact5verices}, $G-x$ admits a prime orientation $P$. Fix $\alpha \in  \overline{\{x,s,s'\}}$, and let $Q$ be an orientation of $G$ such that $Q-x = P$ and $Q(x, \alpha) \neq Q(s', \alpha)$. Such an orientation exists because $s'\alpha \in E(G)$. Set $Z = \overline{\{x\}}$. By construction, it is easily seen $x \notin \langle Z \rangle_Q$, and that for every $z \in Z$, we have $x \notin Z_Q(z)$. It follows that $x \in \text{Ext}_Q(Z)$ (see (\ref{eqpGX})), that is, $Q$ is prime.  Now suppose that $G-xs$ is primarily orientable. Let $\overrightarrow{G-xs}$ be a prime orientation of $G-xs$. Since $G$ is connected and $G(s, xs')=0$, there exists $v \in \overline{\{x,s,s'\}}$ such that $G(v,s) =1$, and hence $G(v,s') =1$ because $ss'$ is a module of $G-x$. Thus, there exists an orientation $\overrightarrow{G}$ of $G$ satisfying the following three conditions :
	\begin{itemize}
		\item $\overrightarrow{G}-xs = \overrightarrow{G-xs}$,
		\item for every $u \in \overline{\{x,s,s'\}}$, $\overrightarrow{G}(u,s) = \overrightarrow{G}(u,s')$,
		\item $\overrightarrow{G}(v,x) \neq \overrightarrow{G}(v,s')$.
		\end{itemize}
Let $X = \overline{xs}$, and consider the partition $p_{(G,\overline{X})}$ (see Notation~\ref{notationpGX}). By construction, we have $s \in X_{\overrightarrow{G}}(s')$, $x \notin X_{\overrightarrow{G}}(s')$, and $ss'$ is not a module of $\overrightarrow{G}$. It follows from Assertion~2 of Lemma~\ref{lemER} that $\overrightarrow{G}$ is prime. 	
  	
 Second, suppose that $G-x$ is sduo-free. It follows from the induction hypothesis that $G-x$ is $K_4$ or $G-x$ is primarily orientable. When $G-x$ is $K_4$, the graph $G$ is primarily orientable by Fact~\ref{fact5verices}. So suppose that $G-x$ is primarily orientable, and let $\overrightarrow{G-x}$ be a prime orientation of $G-x$. We consider the set $\overrightarrow{\mathcal{G}}$ of all the orientations $\overrightarrow{G}$ of $G$ such that $\overrightarrow{G}-x = \overrightarrow{G-x}$. Set $X = \overline{\{x\}}$ and $Y = \{y \in X : xy \in E(G)\}$. Since $G$ is connected, we have $Y \neq \varnothing$. We distinguish the following two cases. 
 
 First suppose $Y = X$. Let us consider the set $\overrightarrow{\mathcal{G}}_0$ (resp. $\overrightarrow{\mathcal{G}}_1$, $\overrightarrow{\mathcal{G}}_2$) of the elements $\overrightarrow{G}$ of $\overrightarrow{\mathcal{G}}$ such that $x \in \text{Ext}_{\overrightarrow{G}}(X)$ (resp. $x \in \langle X \rangle_{\overrightarrow{G}}$, $x \in X_{\overrightarrow{G}}(u)$ for some $u \in X$). Notice that $\overrightarrow{\mathcal{G}}_0$ is the set of prime elements of $\overrightarrow{\mathcal{G}}$, and that $\{\overrightarrow{\mathcal{G}}_0, \overrightarrow{\mathcal{G}}_1, \overrightarrow{\mathcal{G}}_2\}$ form a partition of $\overrightarrow{\mathcal{G}}$ (see (\ref{eqpGX})). We will prove that $\overrightarrow{\mathcal{G}}_0 \neq \varnothing$ by a counting argument. Indeed, it is easily seen that $|\overrightarrow{\mathcal{G}}| = 2^{n-1}$, $|\overrightarrow{\mathcal{G}}_1| = 2$, and $|\overrightarrow{\mathcal{G}}_2| = 2(n-1)$. Since $|\overrightarrow{\mathcal{G}}| = |\overrightarrow{\mathcal{G}}_0| + |\overrightarrow{\mathcal{G}}_1| + |\overrightarrow{\mathcal{G}}_2|$ because $\{\overrightarrow{\mathcal{G}}_0, \overrightarrow{\mathcal{G}}_1, \overrightarrow{\mathcal{G}}_2\}$ is a partition of $\overrightarrow{\mathcal{G}}$, it follows that $|\overrightarrow{\mathcal{G}}_0| = 2^{n-1} -2n$. Since $2^{n-1} -2n > 0$ because $n \geq 5$, we obtain $\overrightarrow{\mathcal{G}}_0 \neq \varnothing$, and hence $G$ is primarily orientable. 
 
 Second, suppose $Y \varsubsetneq X$. Let $O \in \overrightarrow{\mathcal{G}}$. Since $\varnothing \neq Y \varsubsetneq X$, we have $x \notin \langle X \rangle_G$, and hence $x \notin \langle X \rangle_{O}$ (see (\ref{eqmodGGfleche})). Moreover, if $x \in X_{O}(u)$ for some $u \in X$, then since $xu$ is a module of $O$ and thus of $G$ (see (\ref{eqmodGGfleche})), and since $G$ is sduo-free, we have $u \in Y$. It follows from (\ref{eqpGX}) that
 \begin{equation} \label{eqq}
 \text{given} \ O \in \overrightarrow{\mathcal{G},} \ \text{if} \ O \ \text{is decomposable}, \ \text{then} \ x \in X_{\overrightarrow{G}}(u) \ \text{for some} \ u \in Y. 
 \end{equation}
  Now let $\overrightarrow{G}$ be the element of $\overrightarrow{\mathcal{G}}$ such that $\overrightarrow{G}(x,Y)= 1$. If $\overrightarrow{G}$ is prime, then we are done. So suppose that $\overrightarrow{G}$ is decomposable. By (\ref{eqq}), there exists $u \in Y$ such that $x \in X_{\overrightarrow{G}}(u)$. If $Y = \{u\}$, then $G(x, \overline{xu}) =0$, and hence $G(u, \overline{xu}) =0$ because $xu$ is a module of $\overrightarrow{G}$ and thus of $G$ (see (\ref{eqmodGGfleche})). It follows that $xu$ is a connected component of $G$, which contradicts that $G$ is connected. Thus $\{u\} \varsubsetneq Y \varsubsetneq X$. Fix $y \in Y \setminus \{u\}$, and consider the orientation $\overrightarrow{G}'$ of $G$ obtained from $\overrightarrow{G}$ by reversing the arc $\overrightarrow{xy}$ of $\overrightarrow{G}$, that is, $\overrightarrow{G}' = (V(G), (A(\overrightarrow{G}) \setminus \{\overrightarrow{xy}\}) \cup \{\overrightarrow{yx}\})$. 
   For every $v \in Y \setminus \{u\}$, we have $\overrightarrow{G}(u,v) = 1$ because $\overrightarrow{G}(x,v) = 1$ and $xu$ is a module of $\overrightarrow{G}$. Thus $\overrightarrow{G'}(u, Y \setminus \{u\}) = \overrightarrow{G'}(x, Y \setminus \{y\}) = \overrightarrow{G'}(y,x)=1$. Therefore, for every $z \in Y$, we have $x \notin Y_{\overrightarrow{G'}}(z)$. Since $\overrightarrow{G'} \in \overrightarrow{\mathcal{G}}$, it follows from (\ref{eqq}) that $\overrightarrow{G'}$ is prime, which completes the proof.     
\end{proof}

We are now ready to prove our main result. We need the following notation.
\begin{Notation} \normalfont
	Given an oriented graph $O$, for every vertex $v$ of $O$, the vertex set $\{x \in V(O) : \overrightarrow{vx} \in A(O)\}$ is denoted by $N_O^+(v)$.
\end{Notation}            
  
\begin{proof}[Proof of Theorem \ref{thmprincipal}]
	Let $G$ be a graph with at least three vertices, other than $K_4$. To begin, suppose that the graph $G$ is primarily orientable. The graph $G$ is connected (see (\ref{eqconnecprimorient})). Now let $S$ be a s-module of $G$. Consider a prime orientation $\overrightarrow{G}$ of $G$, and consider the function
	\begin{equation*}
	\begin{array}{rccl}
	f : &S&\longrightarrow&2^{N_G(S)}\\
	&x&\longmapsto& f(x)= N_{\overrightarrow{G}}^+(x).
	\end{array}
	\end{equation*}
	Given distinct $x, y \in S$, if $f(x) = f(y)$, then $xy$ is a nontrivial module of the prime oriented graph $\overrightarrow{G}$, a contradiction. Thus $f(x) \neq f(y)$, and hence $f$ is injective. It follows that $|S| \leq |2^{N_G(S)}| = 2^{d_G(S)}$, that is, $d_G(S) \geq \log_2(|S|)$. 
	
	Conversely, suppose that $G$ is connected and that 
	\begin{equation} \label{eqthm}
	d_G(M) \geq \log_2(|M|) \ \text{for every s-module} \ M \ \text{of} \ G. 
	\end{equation}
	When $v(G) = 3$ or $4$, the graph $G$ is primarily orientable by Fact~\ref{factsize34}. So suppose $v(G) \geq 5$. To prove that $G$ is primarily orientable, we proceed by induction on $|\text{sduo}(G)|$. If $|\text{sduo}(G)|=0$, i.e. $G$ is sduo-free, then $G$ is primely orientable by Proposition~\ref{propsduofree}. Now suppose $|\text{sduo}(G)| \geq 1$. Let $S$ be a s-module of $G$ which is of maximum size. Since $|S| \geq 2$ by the definition of a s-module, and since $S \neq V(G)$ because $G$ is connected, then the module $S$ of $G$ is nontrivial. Fix $x \in S$, and consider the graph $G-x$. By Lemmas~\ref{lemrem} and \ref{leminclustrict}, $G-x$ is connected and $\text{sduo}(G-x) \varsubsetneq \text{sduo}(G)$. If $G-x$ is $K_4$, the graph $G$ is primarily orientable by Fact~\ref{fact5verices}. Hence suppose that $G-x$ is not $K_4$. Since $G-x$ is connected, and $|\text{sduo}(G-x)| < |\text{sduo}(G)|$ because $\text{sduo}(G-x) \varsubsetneq \text{sduo}(G)$, to apply the induction hypothesis to $G-x$, it only remains to verify that 
	\begin{equation} \label{eqthm2}
	d_{G-x}(M) \geq \log_2(|M|) \ \text{for every s-module} \ M \ \text{of} \ G-x. 
	\end{equation}
	Let $M$ be a s-module of $G-x$. If $M \cap S \neq \varnothing$ and $M \setminus S \neq \varnothing$, then it is easy to see that $M \cup S$ is a s-module of $G$, which contradicts the maximality of the s-module $S$ of $G$. Therefore, $M \cap S = \varnothing$ or $M \subseteq S \setminus \{x\}$. First suppose $M \subseteq S \setminus \{x\}$. In this instance, $N_G(M) = N_G(S)$ and hence $d_{G-x}(M) = d_{G}(S)$. Since $d_{G}(S) \geq \log_2(|S|)$, it follows that $d_{G-x}(M) \geq \log_2(|S|) > \log_2(|M|)$. Second suppose $M \cap S = \varnothing$. Since $M$ is a s-module of $G-x$, to show that $M$ is also a s-module of $G$, it suffices to verify that for every $m,m' \in M$, we have $G(x,m) = G(x,m')$. Let $m,m' \in M$, and consider $s \in S \setminus \{x\}$. We have $G(x,m) = G(s,m)$ and $G(x,m') = G(s,m')$ because $S$ is a module of $G$. Moreover, $G(s,m) = G(s,m')$ because $M$ is a module of $G-x$. Thus $G(x,m) = G(x,m')$, and hence $M$ is a s-module of $G$. It follows that $d_G(M) \geq \log_2(|M|)$ (see (\ref{eqthm})). Moreover, we have $G(x,M) =0$ or $G(x,M) =1$. In the first instance, we have $d_G(M) = d_{G-x}(M)$, and since $d_G(M) \geq \log_2(|M|)$, we obtain $d_{G-x}(M) \geq \log_2(|M|)$. In the second instance, since $M$ and $S$ are disjoint modules of $G$, we have $G(M,S)=1$ so that $d_G(M) \geq |S|$, and hence $d_{G-x}(M) \geq |S|-1$. Since $|S| \geq |M|$ by maximality of $S$, it follows that $d_{G-x}(M) \geq |M|-1 \geq \log_2(|M|)$. Thus (\ref{eqthm2}) holds.
	
	Now by the induction hypothesis applied to $G-x$, the graph $G-x$ is primarily orientable. Let $\overrightarrow{G-x}$ be a prime orientation of $G$, and consider the function 
	\begin{equation*}
	\begin{array}{rccl}
	g : &S \setminus \{x\}&\longrightarrow&2^{N_G(S)}\\
	&v&\longmapsto& g(v)= N_{\overrightarrow{G-x}}^+(v).
	\end{array}
	\end{equation*}
	Since $|S \setminus \{x\}| < 2^{d_G(S)}$ because $d_G(S) \geq \log_2(|S|)$, and since $|2^{N_G(S)}| = 2^{d_G(S)}$, we have $|S \setminus \{x\}| < |2^{N_G(S)}|$. Therefore, the function $g$ is not surjective. It follows that $2^{N_G(S)} \setminus \{g(v) : v \in S \setminus \{x\}\}) \neq \varnothing$. So let $Y \in 2^{N_G(S)} \setminus \{g(v) : v \in S \setminus \{x\}\}$, and consider the orientation $\overrightarrow{G}$ of $G$ defined by $\overrightarrow{G}-x = \overrightarrow{G-x}$ and $N_{\overrightarrow{G}}^+(x) = Y$. Set $X = \overline{\{x\}}$. Since $x \notin \langle X \rangle_G$ because $G(x, S \setminus \{x\}) = 0 \neq G(x, N_G(S)) =1$, then $x \notin \langle X \rangle_{\overrightarrow{G}}$ (see (\ref{eqmodGGfleche})). Suppose for a contradiction that $\overrightarrow{G}$ is decomposable. Since $x \notin \langle X \rangle_{\overrightarrow{G}}$, there exists $u \in X$ such that $x \in X_{\overrightarrow{G}}(u)$ (see (\ref{eqpGX})). If $u \in S \setminus \{x\}$, then since $xu$ is a module of $\overrightarrow{G}$ and $\overrightarrow{G}(x,u)=0$, we have $N_{\overrightarrow{G}}^+(x) = N_{\overrightarrow{G}}^+(u) = N_{\overrightarrow{G}-x}^+(u) = g(u)$, which contradicts $N_{\overrightarrow{G}}^+(x) = Y \notin \{g(v) : v \in S \setminus \{x\}\}$. Therefore $u \notin S$. Moreover, if $u \in N_G(S)$, then for $s \in S \setminus \{x\}$, we have $G(s,x) =0 \neq G(s,u) =1$, contradicting that $xu$ is a module of $\overrightarrow{G}$ and thus of $G$ (see (\ref{eqmodGGfleche})). Thus $u \in X \setminus (S \cup N_G(S))$. It follows that $xu$ is a s-duo of $G$. Since $S$ and $xu$ are non-disjoint s-modules of $G$, then $S \cup xu = S \cup \{u\}$ is also a s-module of $G$ (see Remark~\ref{remunionsmodules}), which contradicts the maximality of $S$. Thus $\overrightarrow{G}$ is prime, completing the proof.           
\end{proof}  
  
 {}
 
\end{document}